\documentclass[11pt]{amsart}   
\usepackage{latexsym}
\usepackage{amsfonts}
\usepackage{amscd}
\usepackage[all]{xypic}
\theoremstyle{plain}

\newtheorem{theor}{Theorem}[section]
\newtheorem{lem}[theor]{Lemma}
\newtheorem{co}[theor]{Corollary}

\theoremstyle{definition}

\newtheorem{example}[theor]{Example}
\theoremstyle{remark}
\newtheorem{re}[theor]{Remark}
\numberwithin{equation}{section}

\begin{document}
\noindent                                             
\begin{picture}(150,36)                               
\put(5,7){\textbf{Topology Proceedings}}              
\put(0,0){\framebox(140,34)}                          
\put(2,2){\framebox(136,30)}                          
\end{picture}                                         

\vspace{0.5in}

\title[Group homomorphisms induced by isometries]%
{Group homomorphisms induced by isometries of spaces of almost
periodic functions}

\author{Salvador Hern\'{a}ndez}
\address{Departamento de Matem\'aticas.\newline Universidad
Jaume I,\newline 12071-Castell\'on,\newline Spain.}
\email{hernande@mat.uji.es}
\thanks{Research partially supported by Spanish DGES, grant BFM2000-0913,
 and Generalitat Valenciana, grant CTIDIB/2002/192.}



\begin{abstract}
Let $G$ and $H$ be locally compact groups and consider their associate
spaces of almost periodic functions $AP(G)$ and $AP(H)$. We investigate the
continuous group homomorphisms induced by isometries of $AP(G)$ into
$AP(H)$. Among others, the following results are proved:

{\bf Theorem} Let $G$ and $H$ be $\sigma$-compact maximally almost
periodic locally compact groups. Suppose that $T$ is a
non-vanishing linear isometry of $AP(G)$ into $AP(H)$ that
respects finite dimensional unitary representations. Then there is
a closed subgroup $H_0\subseteq H$, a continuous group
homomorphism $t$ of $H_0$ onto $G$ and an character $\gamma\in
\widehat{H}$ such that $(Tf)(h)=\gamma (h)~f(t(h))$ for all $h\in
H_0$ and for all $f\in C(G)$.

{\bf Theorem} Let $G$ and $H$ be $LC$ Abelian groups and $H$ is
connected. Suppose that $T$ is a non-vanishing linear isometry of $AP(G)$ into $AP(H)$
that preserves trigonometric polynomials. Then there is a closed
subgroup $H_0\subseteq H$, a continuous group homomorphism $t$ of $H_0$ onto
$G$, an element $h_0\in H_0$, a character $\alpha \in \widehat{H}$ and an
unimodular complex number $a$ such that
$(Tf)(h)=a\cdot \alpha (h)~\cdot f(t(h-h_0))\text{ for
 all }h\in H_0\text{
and for all }f\in C(G)\text{.}$
\end{abstract}
\maketitle

\section{Introduction}

The deduction of topological (resp. algebraic) links between
topological groups $G$ and $H$ from the existence of certain
functional analytic relationships between appropriate spaces of
mappings defined on the groups has been widely studied in the
literature. For instance, the rings of all continuous functions or
the spaces of all integrable functions with respect to the Haar
measures of locally compact groups have been considered in that
direction. See (cf. \cite{degh, edwa, gali, rudi}) and the
references in those papers.

Here, we deal with isometries defined between certain spaces of continuous
functions 
in order to investigate 
when these isometries induce continuous 
group homomorphisms (defined
between subgroups of $H$ and $G$) representing the isometries canonically.

Our start point has been the following theorem of Holszty\'{n}ski
(cf. \cite{hols}) that is an extension of the well-known
Banach-Stone theorem. Let us say that a linear map
$T:C(X)\longrightarrow C(Y)$ is \textit{canonical}  when
$T(f)=w\cdot (f\circ t)$ for all $f\in C(X)$, where $w$ belongs to
$C(Y)$ and $t$ is a continuous map from $Y$ into $X$.

\begin{theor}
\label{th01} {\bf [Holszty\'{n}ski]} Let $X$ and $Y$ be compact
spaces and let $T$ be an isometry of $C(X)$ into $C(Y)$, then
there exists a closed subspace $Y_{0}$ of $Y$ and 
a canonical map $T_0$ of $C(X)$ onto $C(Y_0)$ such that the
following diagram commutes
\[
\xymatrix{ C(X) \ar@{>}[rr]^T \ar[dr]^{T_0} & & C(Y)\ar[dl]_{R} \\
& C(Y_0) &}
\]
here, $R$ denotes the restriction mapping from $C(Y)$ onto
$C(Y_0)$.
\end{theor}

Section 2 is devoted to study the continuous group homomorphisms that arise out of isometries
defined between the spaces of continuous functions of two compact groups.
Here, a basic tool 
has been the notion of group representation 
and the well-known fact that the set of all 
continuous functions on an compact group is the uniform closure of the set of trigonometric polynomials defined on it.

In Section 3 we deal with the extension of the results obtained for compact groups
to a more general context. We show how, in
many cases, the compactness constraint hold on the groups involved may be relaxed if we consider
the space of almost periodic functions defined on them and impose an
additional natural requirement on the isometry. Namely, the isometry must preserve non-vanishing functions.
In fact, a main goal in this section has been to obtain a variant of the Tannaka-Kre\u{\i}n duality theory
for general maximally almost periodic locally compact (MAPLC) groups using the Banach space structure of the set of almost periodic functions.

\section{Compact Groups}

In this section,we are basically concerned with the following question: Assuming that we are
dealing with compact groups, which additional
hypothesis must be imposed on the isometry of Theorem \ref{th01} in order to obtaining continuous
homomorphisms instead of "plain" continuous mappings? Obviously, something has to be assumed
on the isometry since each homeomorphism, say $t$, gives place to the isometry $Tf=f\circ t$.
Here, we are going to consider isometries which behave nicely with respect
to the notion of group representation. Firstly, since our concern
is to introduce a variant of the Tannaka-Kre\u{\i}n duality theory
for general MAP locally compact groups, we recall here some basic
definitions about the Kre\u{\i}n algebra associated to any compact
group and the Tannaka-Kre\u{\i}n duality theory.

Suppose that $K$ is a (Hausdorff) compact group and let $\Sigma
=\Sigma (K)$ denote the set of equivalence classes of continuous
unitary irreducible representations of $K$. Choose a member of
each $\sigma \in \Sigma $, write it in matrix form $V^{\sigma
}=(v_{jk}^{\sigma })_{j,k=1}^{d(\sigma )}$ with respect to some
orthonormal basis, and let $B=B(K)$ consists of all the
''coordinates'' functions on $G$ so formed. In other words, $B$ is
the set of functions $x\longmapsto v_{jk}^{\sigma }(x)$, where
$\sigma \in \Sigma $ and $1\leq j,k\leq d(\sigma )$. The linear
span of these functions is independent of the selected members of
each $\sigma $ and the particular orthonormal bases. It is called
the space of \textit{trigonometric polynomials} on $K$ and is
closed under pointwise multiplication and complex-conjugation.
Denote it by $\mathcal{T}(K)$. The triple
$(\mathcal{T}(K),B,\Sigma )$ is a Kre\u{\i}n algebra
representative of $K$. In short, the Tannaka-Kre\u{\i}n theory
establishes that if $G_1$ and $G_2$ are compact groups that are
not topological isomorphs then they may not have isomorphic
Kre\u{\i}n algebras representatives. (See \cite[Section 30]{hr2},
\cite[V.4]{heyer} and \cite{frlupr} to find further information on
this topic.)

Let $\mathcal{U}(n)$ denote the {\em unitary group of order $n$},
namely, the group of all complex-valued $n\times n$ matrices $A$
for which $A^{-1}=\overline{A}^{*}$, i.e., the conjugate transpose
of $A$. 
In general, a {\em
unitary representation} $D$ of a (topological) group $G$ is a
(continuous) homomorphism into the group of all isometric linear
endomorphisms of a complex Hilbert space $\mathcal{E}$ (the latter
equipped with the topology defined by its inner product, see {\sc
Naimark}~\cite{naimark} (Chapter IV)). The space $\mathcal{E}$ is
the \textit{representation space of D}. When
dim~$\mathcal{E}<\infty$, we say that $D$ is a {\it finite
dimensional representation}; in this case, $D$ is a homomorphism
into one of the groups $\mathcal{U}(n)$. If $G$ is a topological
group, the symbol $Rep_n(G)$ denotes the set of all
representations of $G$ into $\mathcal{U}(n)$:
\[
Rep_n(G)=\{f:G\longrightarrow \mathcal{U}(n)\quad |\quad \hbox{%
$f$ is a continuous homomorphism}\}
\]
Now, suppose that $G$ and $H$ are compact Hausdorff groups and let
$T$ be a linear map of $C(G)$ into $C(H)$, we say that $T$
\textit{respects unitary representations} when for all $n<\omega $
and $D=(d_{jk})_{j,k=1}^{n}\in Rep_n(G)$, the matrix
$T(D)=(T(d_{jk}))_{j,k=1}^{n}$ defines a continuous representation
of $H$. Next we prove that the continuous mappings associated to
this kind of isometries preserve the algebraic structure of the
groups concerned.

\begin{theor}
\label{th21}Let $G$ and $H$ be compact Hausdorff groups and let $T$
be a linear isometry of $C(G)$ into $C(H)$
respecting unitary representations. Then there is a closed subgroup $H_{0}\subseteq H$, a
continuous group homomorphism $t$ of $H_{0}$ onto $G$ and an element
$\gamma \in \widehat{H}$ such that
\begin{equation*}
(Tf)(h)=\gamma (h)~f(t(h))\text{ for all }h\in H_{0}\text{ and for all }f\in
C(G)\text{.}
\end{equation*}
\end{theor}
\begin{proof}
There is no loss of generality in assuming that $T(1)=1$. Otherwise,
since $T$ respects unitary representations, we have that $T(1)=\gamma $ must be a
linear representation (that is, a continuous group homomorphism
of $H$ into $\mathbb{T}$, the unity circle of the complex plane).
Then the map $R=\gamma ^{-1}\cdot T$
also is an isometry of $C(G)$ into $C(H)$ that respects unitary
representations and $R(1)=1$.

By Theorem \ref{th01} there is a closed subset $H_{0}\subseteq H$,
a continuous mapping $t$ of $H_{0}$ onto $G$ and a unimodular
continuous map $w:H_{0}:\longrightarrow \mathbb{C}$ such that
$(Tf)_{|H_{0}}=w\cdot (f\circ t)$ for all $f\in C(G)$. Hence, if
$f\equiv 1$ on $G$, we have that $1=T(1)=w\cdot1 $. That is,
$w\equiv 1$ on $H_{0}$. Furthermore, for $f$ and $g$ in $C(G)$, we
have that $T(f\cdot g)_{|H_{0}}=(f\cdot g)\circ t=(f\circ t)\cdot
(g\circ t)= T(f)_{|H_{0}}\cdot T(g)_{|Y_{0}}$. Since $t$ is an
onto map, it follows that $T$ is a multiplicative isometry of
$C(G)$ into $C^{\ast}(H_{0})$. On the other hand, according to
(Holszty\'{n}ski) Theorem \ref{th01}, the mapping $t$ is defined
as follows: for every point $p\in G$ define $C_{p}=\{f\in
C(G):\parallel~f~\parallel =~|f(p)|\}$ and for $f\in C(G)$ set
$L(f)=\{q\in H:||T(f)||=~|T(f)(q)|\}$. If $I_{p}=\cap \{L(f):f\in
C_{p}\}$ then $I_{p}$ is a non empty closed subset of $H$,
$H_{0}=\cup_{p\in G} I_{p}$ and $t(q)=p$ for all $q\in I_{p}$ (cf.
\cite{hols}). Using this fact, we now check that $H_0$ is a
subgroup of $H$ and $t$ is a group homomorphism.

In order to do it, take two points $x,y\in H_0$ and 
denote by $\delta _{xy}$ the point mass evaluated at ${xy}$. We claim that
there is a singleton $a\in G$ such that
$\delta _{xy}\circ T$ coincides on $C(G)$ with the point mass $\delta_{a}$.
In fact, after the Tannaka-Kre\v{\i}n duality theorem (cf. \cite[(30.30)]{hr2}),
it is enough to verify that $\delta _{xy}\circ T$
is a multiplicative linear functional on $\mathcal{T}(G)$ such that
$(\delta _{xy}\circ T)(\widetilde{f})=\overline{(\delta _{xy}\circ T)(f)}$
for all $f\in \mathcal{T}(G)$.

Let us see that $\mathcal{\delta }_{xy}\circ T$ is multiplicative.
According to \cite[(30.26)]{hr2}, we must prove that
$(\delta _{xy}\circ T)(D\otimes E)=
(\delta _{xy}\circ T)(D)\otimes (\delta_{xy}\circ T)(E)$
for all $D \in Rep_{n}(G), E \in Rep_{m}(G)$ and $n, m<\omega$. Now,
$$
(\delta _{xy}\circ T)(D\otimes E)
=\delta _{xy}(T(D\otimes E))
=T(D\otimes E)(xy)
$$
and, since $T$ respects unitary representations, we have that
\begin{equation*}
\begin{array}{ll}
T(D\otimes E)(xy) &=  \\
T(D\otimes E)(x)\cdot T(D\otimes E)(y) &= \\
(\delta _{x}\circ T)(D\otimes E)\cdot (\delta _{y}\circ T)(D\otimes E)\text{.}&
\end{array}
\end{equation*}
On the other hand, $(\delta _{x}\circ T)$ and $(\delta _{y}\circ T)$ are
multiplicative linear functionals on $\mathcal{T}(G)$ because $x$ and $y$
belong both to $H_{0}$ and $(Tf)_{|H_{0}}=f\circ t$ for all $f\in C(G)$.
That is $(\delta_{x}\circ T)(D)=(TD)(x)=D(t(x))=\delta_{t(x)}(D)$ for all $x\in H_0$, $D\in Rep_{n}(G)$,
and $n<\omega$. Hence 

\begin{equation*}
\begin{array}{ll}
(\delta _{xy}\circ T)(D\otimes E) &= \\
(\delta _{x}\circ T)(D\otimes E)\cdot (\delta _{y}\circ T)(D\otimes E) &= \\
\lbrack \delta _{t(x)}(D\otimes E)]\cdot \lbrack \delta _{t(y)}(D\otimes E)]
&= \\
(D\otimes E)(t(x))\cdot (D\otimes E)(t(y)) &= \\
(D\otimes E)(t(x)t(y)) &= \\
D(t(x)t(y))\otimes E(t(x)t(y)) &= \\
\lbrack D(t(x))\cdot D(t(y))]\otimes \lbrack E(t(x))\cdot E(t(y))] &= \\
\lbrack \delta _{t(x)}(D)\cdot \delta _{t(y)}(D)]\otimes \lbrack \delta
_{t(x)}(E)\cdot \delta _{t(y)}(E)] &= \\
\lbrack (\delta _{x}\circ T)(D)\cdot (\delta _{y}\circ T)(D)]\otimes
\lbrack (\delta _{x}\circ T)(E)\cdot (\delta _{y}\circ T)(E)] &= \\
\lbrack (TD)(x)\cdot (TD)(y)]\otimes \lbrack (TE)(x)\cdot (TE)(y)] &= \\
(TD)(xy)\otimes (TE)(xy) &= \\
\delta _{xy}(T(D))\otimes \delta _{xy}(T(E)) &= \\
(\delta _{xy}\circ T)(D)\otimes (\delta _{xy}\circ T)(E)\text{.}&
\end{array}
\end{equation*}
\newline
The equality
\begin{equation*}
(\delta _{xy}\circ T)(\widetilde{f})=\overline{(\delta _{xy}\circ T)(f)}
\end{equation*}
for all $f\in \mathcal{T}(G)$ is handled similarly using \cite[(30.2)]{hr2}.
By the Tannaka-Kre\v{\i}n duality theorem (cf. \cite[(30.30)]{hr2}),
we have proved that $\delta _{xy}\circ T$ coincides with a point mass evaluated
at a point $a\in G$. 
Thus, $(Tf)(xy)=(\delta _{xy}\circ~T)(f)=f(a)$ for all $f\in C(G)$. 
By the way in which $t$ was defined above, it follows that
$xy\in I_{a}$, $a\in H_0$ and $t(xy)=a$.
Hence, $(Tf)(xy)=f(t(xy))$ for all $f\in C(G)$.

On the other hand, since $T$ respects unitary representations,
for every $D\in Rep_{n}(G)$, we have that
\begin{equation*}
\begin{array}{ll}
D(t(xy))=(d_{jk}(t(xy)) &= \\
(T(d_{jk})(xy)) &= \\
(TD)(xy) &= \\
(TD)(x)(TD)(y) &= \\
(T(d_{jk})(x))(T(d_{jk})(y)) &= \\
D(t(x))D(t(y)) &= \\
D(t(x)t(y)). &
\end{array}
\end{equation*}
\qquad For compact groups, unitary representations separate points, therefore,
$t(xy)=t(x)t(y)$. 
Analogously, it is verified that for every $x\in H_{0}$, it holds that
$x^{-1}\in H_0$ and $t(x^{-1})=t(x)^{-1}$. Thus, we conclude that
$H_0$ is a subgroup of $H$ and $t$ is a continuous group homomorphism of $H_0$ onto $G$.
This completes the proof.
\end{proof}
An easy consequence of Theorem \ref{th21} is the following result

\begin{co}
\label{co22}Let $G$ and $H$ be two compact groups and suppose
that $T$ is a linear isometry of $C(G)$ onto $C(H)$
that respects unitary representations. Then there exists a topological isomorphism $t$ of $H$ onto $G$ and an element $\gamma \in \widehat{H}$ such that
\begin{equation*}
(Tf)(h)=\gamma (h)~f(t(h))\text{ for all }h\in H\text{ and for all }f\in
C(G)\text{.}
\end{equation*}
Hence, $G$ and $H$ are isomorphic compact groups.
\end{co}

In general the subgroup $H_{0}$ in Theorem \ref{th21} need not be equal to
$H$. Indeed, if we take
$G$ as the group $\{-1,1\}$ and $H$ equal to $\Bbb{T}$, then no isometry
of $C(G)$ into $C(H)$
may be represented by a continuous mapping of $H$ onto $G$.

When the groups are Abelian and connected, the results above can be
improved considerably.
It is well known that every unitary representation of a $LC$ Abelian group is unitary equivalent
to a direct sum of $1$-dimensional unitary representations and, as a consequence, one replaces
the representation space $Rep(G)$ by the dual group $\widehat{G}$
of all continuous characters on $G$ (irreducible elements of $Rep_1(G)$).
Thus, the ring of trigonometric polynomials, $\mathcal{T}(G)$,
is generated by $\widehat{G}$ in this case.
Given two compact groups $G$ and $H$ and an isometry $T$
of $C(G)$ into $C(H)$, we say that $T$
{\it preserves trigonometric polynomials } when
for every $\phi\in \mathcal{T}(G)$ 
it holds that $T(\phi)\in \mathcal{T}(H)$. 

\begin{theor}
\label{th23}Let $G$ and $H$ be compact groups, with $H$ connected. Suppose
that $T$ is a linear isometry of $C(G)$ into $C(H)$ that preserves
trigonometric polynomials.
Then there is a closed subgroup $H_{0}\subseteq H$, a continuous group homomorphism
$t$ of $H_{0}$ onto $G$,
an element $h_{0}\in H_{0}$, a character $\alpha \in \widehat{H}$ and an
unimodular complex number $a$ such that
\begin{equation*}
(Tf)(h)=a\cdot \alpha (h)~\cdot f(t(h-h_{0}))\text{ for all }h\in H_{0}\text{
and for all }f\in C(G)\text{.}
\end{equation*}
Moreover, if $T$ is an onto isometry then $H_{0}=H$ and, as a consequence,
$G$ and $H$ are topologically isomorphic.
\end{theor}
\begin{proof}
Since $T(1)$ is an unimodular trigonometric polynomial, we know
(\cite[(1.1)]{glick2}) that there is a character $\alpha \in \widehat{H}$ and
a unimodular complex number $a$ such that $T(1)=a\cdot \alpha $. Set
$R=T(1)^{-1}\cdot T$. It is clear that $R(1)=1$ and $R$ preserves
trigonometric polynomials. By Theorem \ref{th01}
there is a closed subset $Y\subseteq H$, a continuous map $r$ of $Y$ onto $G$ and an element
$w\in C^{\ast }(Y)$, $\mid w\mid \equiv 1$, such that $(Rf)(y)=w(y)~f(r(y))$
for all $y\in Y$ and $f\in C(G)$.
On the other hand, since $R(1)=1$, it follows that $w\equiv 1$ on $Y$.
So that $R(f)_{|Y}=f\circ r$ for all $f\in C(G)$.

Now, choose $h_0\in Y$ with $r(h_0)=1_G$, the neutral element in $G$. If we define
$S(f)=(R(f))_{h_0}$; that is $(R(f))_{h_0}(h)=(Rf)(h+h_0)$ for all $h\in H$
and $f\in C(G)$. Then $S$ is an isometry
of $C(G)$ into $C(H)$ that preserves trigonometric polynomials and with $S(1)=1$.
Moreover, defining $Y_0 =\{y-h_0 : y\in Y\}$, it holds that $(Sf)_{|Y_0}=f\circ s$,
where $s$ is a continuous mapping of $Y_0$ onto $G$ defined by $s(h)=r(h+h_0 )$.
Therefore, $1_H \in Y_0$ and $s(1_H )=1_G$.
Notice that, when $\chi \in \widehat{G}$, (\cite[(1.1)]{glick2}) yields that
$S(\chi )=b(\chi )\cdot \gamma (\chi )$
with $\gamma (\chi )\in \widehat{H}$ and $b(\chi )$ an unimodular complex
number. So that $(b(\chi )\cdot \gamma (\chi ))_{|Y_0}=\chi \circ s$.
Then $b(\chi )\cdot \gamma (\chi )(1_H )=\chi (1_G )=1$. Thus, $b(\chi )=1$
and, as a consequence, $S(\chi )=\gamma (\chi )\in \widehat{H}$.
In other words, $S$ is multiplicative isometry of $C(G)$ into $C(H)$ that preserves
characters.
Thus, we are in position to apply Theorem \ref{th21}.
Hence, there exists a closed subgroup $H_{0}\subseteq H$, a continuous
group homomorphism $t$ of $H_{0}$ onto $G$ and an character $\beta \in \widehat{H}$ such that
\begin{equation*}
(Sf)(h)=\beta (h)~f(t(h))\hbox{ for all }h\in H_{0}\hbox{ and for all }f\in C(G)\hbox{.}
\end{equation*}
Hence,
\begin{equation*}
\begin{array}{ll}
T(f)(h)= a\cdot \alpha (h)\cdot (Rf)(h) &= \\
a\cdot \alpha (h)\cdot (Sf)(h-h_0) &= \\
a\cdot \alpha (h)\cdot \beta (h)\cdot f(t(h-h_0)) &= \\
a\cdot \lambda (h)\cdot f(t(h-h_0)).&
\end{array}
\end{equation*}
where $\lambda =\alpha \cdot \beta \in \widehat{H}$.

Finally, when $T$ is a onto isometry, it is clear that $H_{0}$ coincides with $H$
and, therefore, $t$ is a topological isomorphism. This completes the proof.
\end{proof}
\begin{co}
\label{co24}Let $\rho :H\longrightarrow G$ where $G$ and $H$ are compact Abelian
groups and $H$ is connected. Then $\rho $ preserves trigonometric
polynomials iff $\rho =t+\theta $ where $t$ is a continuous group homomorphism
and $\theta $ is a constant map.
\end{co}
\begin{proof}
\textit{Necessity} is obvious.

\noindent \textit{Sufficiency.} Set $T:C(G)\longrightarrow C(H)$ defined by
$T(f)(y)=f(\rho (y))$. It is easily checked that $T$ is a multiplicative
linear isometry of $C(G)$ into $C(H)$ that preserves trigonometric polynomials.
Applying Theorem \ref{th21}, there is a closed subgroup $H_{0}\subseteq H$, a
continuous group homomorphism $t$ of $H_{0}$ onto $G$,
an element $h_{0}\in H_{0}$, a character $\alpha \in \widehat{H}$ and an
unimodular complex number $a$ such that
\begin{equation*}
(Tf)(h)=a\cdot \alpha (h)~\cdot f(t(h-h_{0}))\text{ for all }h\in H_{0}\text{
and for all }f\in C(G)\text{.}
\end{equation*}
Now, since $T$ is multiplicative, it follows that $\alpha \equiv 1=a$. On
the other hand, because of the way in which $T$ was defined it is clear that
$H_{0}=H$. 
Hence $\rho(h)=t(h-h_{0})=t(h)-t(h_{0})$ for all $h\in H$.
\end{proof}
\section{$\sigma$-compact Locally Compact Groups}
The aim of this section is to extend to maximally almost periodic
$\sigma$-compact locally compact groups
the results obtained previously for compact groups. 

For a topological group $G$, $AP(G)$ denotes the set of all almost
periodic functions on $G$; that is, $AP(G)$ consists of all
complex-valued functions $f$ defined on $G$ such that for every
$\epsilon >0$ there is a finite decomposition
$G=\cup_{i=1}^{n}G_{i}$ with $|f(zxw)-f(zyw)|<\epsilon $ for all
$z,w\in G$ and $x,y\in G_{i}$, $i=1,...,n$. With every topological
group $G$ there is associated a compact group $bG$, the so-called
\textit{Bohr compactification} of $G$, and a continuous
homomorphism $b$ from $G$ onto a dense subgroup of $bG$. Among
such compact groups and continuous homomorphisms, $bG$ and $b$ are
determined by this property: for every continuous homomorphism $h$
from $G$ into a compact group $K$ there is a continuous
homomorphism $\overline{h}$ from $bG$ onto $K$ such that
$h=\overline{h}\circ b$, that is, making the following diagram
commutative:
\[
\xymatrix{ G \ar@{>}[rr]^b \ar[dr]^h & & bG\ar[dl]_{h^b} \\ & K &}
\]
The group $b(G)$ receives a totally bounded group topology
inherited from $bG$. It is usually called the Bohr topology of
$G$. (See \cite[V.4]{heyer} for a full examination of $bG$ and its
properties.) It is known that every finite-dimensional continuous
unitary representation of $G$ extends to continuous representation
on $bG$ (cf. \cite[V.4]{heyer}). As a consequence, there is no
loss of generality in identifying the representations spaces of
$G$ and $bG$. In the sequel, this identification is always
implicitely assumed. The set of all almost periodic functions
$AP(G)$ on a group $G$ coincide with the restrictions to $G$ of
the continuous functions defined on $bG$. Hence, if we denote by
$G^{b}$ the group $G$ equipped with the topology inherited from
$bG$ (the Bohr topology), we have that $AP(G)$ is a Banach
subalgebra of $C^{\ast}(G^{b})$.
Among the representations of $G$, the linear representations, i.e. with its degree $d(\sigma)=1$, form a group $\widehat{G}$ under multiplication. 
The group $\widehat{G}$, equipped with the compact open topology, is called the {\it dual group} of $G$. 
Next we shall apply the methods of the
Section above to investigate how the Banach algebra of all almost periodic functions
determines the topological and algebraic structure of maximally almost periodic locally compact groups.
First, we shall need to recall a few definitions and the following simple lemma.
A topological group $G$ is said to \textit{respect compactness }(cf. \cite{tri1}) when any subset of $G$, which is compact in the Bohr topology, is also compact in the original topology of $G$. 
Given a topological group $G$, by $G^{b}$ we mean the same
algebraic group $G$ equipped with the Bohr topology. For $G$ and
$H$, and $T$ a linear map of $AP(G)$ into $AP(H)$, we say that $T$
is {\it non-vanishing} when $(Tf)(y)\not=0$ for all $y\in H$ if
and only if $f(x)\not=0$ for all $x\in G$. In the sequel all
groups are assumed to be maximally almost periodic.
\begin{lem}
\label{le31} Let Let $G$ and $H$ be $LC$ groups and let $T$ be a
non-vanishing linear isometry of $AP(G)$ into $AP(H)$ that
respects (finite dimensional) unitary representations. Then, for
every $f\in AP(G)$, $\|f\|=|f(x)|$ for some $x\in G$ if and only
if $\|T(f)\|=|(Tf)(y)|$ for some $y\in H$.
\end{lem}
\begin{proof}
There is no loss of generality in assuming that $T(1)=1$. Otherwise,
since $T$ respects unitary representations, we have that $T(1)=\gamma $ must be a
linear representation (that is, a continuous group homomorphism
of $H$ into $\mathbb{T}$, the unity circle of the complex plane).
Then the map $R=\gamma ^{-1}\cdot T$
also is an isometry of $AP(G)$ into $AP(H)$ that respects unitary
representations and $R(1)=1$.

Now, observe that $\|f\|=|f(x)|$, for $x\in G$, if and only if there is a scalar
$\lambda$ with $|\lambda|=\|f\|$ such that $(\lambda - f)(x)=0$.
Since $T$ is non-vanishing,
the latter is equivalent to $(\lambda - Tf)(y)=0$ for $y\in H$ and, therefore,
$\lambda = \|Tf\|=|(Tf)(y)|$.
\end{proof}
Next theorem permits to extend to $LC$ groups the results proved in Section 2 for
compact groups.
\begin{theor}
\label{th32}Let $G$ and $H$ be $MAPLC$ groups such that $H$ is
$\sigma$-compact and $G$ is either $\sigma$-compact or $G$
respects compactness and $G^b$ is realcompact. If $T$ is a
non-vanishing linear isometry of $AP(G)$ into $AP(H)$
that respects (finite dimensional)
unitary representations. Then there is a closed subgroup $H_{0}\subseteq H$, a
continuous group homomorphism $t$ of $H_{0}$ onto $G$ and an element
$\gamma \in \widehat{H}$ such that
\begin{equation*}
(Tf)(h)=\gamma (h)~f(t(h))\text{ for all }h\in H_{0}\text{ and for all }f\in
AP(G)\text{.}
\end{equation*}
\end{theor}
\begin{proof}
Firstly, we shall prove that there is a closed subgroup $H_0\subset H$ and a
group homomorphism $t$ of $H_{0}\subset H^{b}$ onto $G^b$ which is (Bohr) continuous.
Since $AP(G)$ and $AP(H)$ can be identified with $C(bG)$ and $C(bH)$, respectively,
we can apply Theorem \ref{th21} to obtain
a closed subgroup $L_{0}$ of
$bH$, a continuous group homomorphism $t$ of $L_{0}$ onto $bG$ and an element
$\gamma \in \widehat{H}$ such that
$(Tf)(q)=\gamma(q)~f(t(q))$ for all $q\in L_{0}$
and for all $f\in C(bG)$. Moreover, according to (Holszty\'{n}ski) Theorem \ref{th01},
the homomorphism $t$ is defined as follows: for every
point $p\in bG$ define $C_{p}=\{f\in C(bG):\parallel~f~\parallel =~|f(p)|\}$ and
for $f$ any element of $C(bG)$ set $L(f)=\{q\in bH:||T(f)||=~|T(f)(q)|\}$. If
$I_{p}=\cap \{L(f):f\in C_{p}\}$ then $I_{p}$ is a non empty closed subset of
$bH$ and $t(q)=p$ for all $q\in I_{p}$ (cf. \cite{hols}).
Defining $H_0$ to be $H\cap L_0$, in order to prove that $t_{|H_0}$ is an homomorphism
of $H_0$ onto $G$, it is enough to show that
$I_{x}\cap H\neq \emptyset $ for all $x\in G$ and that
$t(H_{0})\subseteq G$.

Pick any point $x\in G$. Every function $f\in C_x$ 
satisfies that $\|f\|=|f(x)|$ and, by Lemma \ref{le31}, there is $y\in H$
such that $\|Tf\|=|(Tf)(y)|$.
Thus, $L_H(f)=L(f)\cap H\not=\emptyset $ for all $f\in C_x$.
We claim that the collection $\{L_H(f):f\in C_x\}$ has the countable intersection property on $H$.

Indeed, let $\{f_n\}_{n<\omega }$ a sequence contained in $C_x$.
Since $L(f)=L(\lambda f)$ for all $\lambda \in \mathbb{C}$, there
is no loss of generality if we assume that $|f_n(x)|=||f_n||=1$
for all $n<\omega $. Define $f=\sum_{n<\omega
}2^{-n}~\overline{f_n(x)}~f_n$. We have that $1=f(x)\leq ||f||\leq
\sum_{n<\omega }2^{-n}~||f_n||\leq 1$ and, therefore, $f\in C_x$.
Hence, by Lemma \ref{le31}, there is $y\in H$ with
$1=||T(f)||=|T(f)(y)|= |\sum_{n< \omega
}2^{-n}~\overline{f_n(x)}~T(f_n)(y)|$, where the continuity of $T$
has been applied at this point. Since
$|2^{-n}~\overline{f_n(x)}~T(f_n)(y)|\leq 2^{-n}$, it follows that
$|T(f_n)(y)|=1$ for all $n<\omega $ so that $y\in \cap _{n<\omega
}L_G(f_n)$. Thus, $\{L_H(f):f\in C_x\}$ is a collection of closed
subset of $H^b$ with the countable intersection property. Now, the
group $H$ is $\sigma$-compact. Since the topology of $H^b$ is
weaker than the topology of $H$, it follows that $H^b$ is
$\sigma$-compact and, therefore, has the Lindel\"{o}f property.
Thus, $\cap \{L_H(f):f\in C_x\}\neq \emptyset $. This proves that
$I_x\cap H\neq \emptyset $ for all $x\in G$.

Suppose now that $y\in H_0$ but $t(y)\notin G$.
Using the work of Ross and Stromberg (cf. \cite{rost})
and Blair (cf. \cite{blair}) or, equivalently, the results proved independently by
\v{S}\v{c}epin (cf. \cite{scepin1}, \cite{scepin2}), we obtain that $G^b$ is
$z$-embedded in $bG$. Since $G^b$ is 
realcompact, 
it is easy to find a function
$f\in C(bG)$ such that 
$f(t(y))=0$ and $f$ does not vanish on $G$. Hence, $f_{|G}\not=0$ and belongs to $AP(G)$.
On the other hand, $T(f_{|G})(y)=w(y)~f(t(y))=0$.
This is a contradiction implying that $t(H_0)\subseteq G$.

Thus, we have proved that $t$ is a (Bohr) continuous group homomorphism
of $H_{0}\subset H^{b}$ onto $G^b$ satisfying that
\begin{equation*}
(Tf)(h)=\gamma (h)~f(t(h))\text{ for all }h\in H_{0}\text{ and for all }f\in
AP(G)\text{.}
\end{equation*}
We shall now prove the continuity of $t$ with respect to the locally compact topologies
of $H_0$ and $G$.

Suppose that $G$ is $\sigma$-compact and let $U$ be an arbitrary element of \ $\mathcal{N}_{G}(1_{G})$ and take $V\in
\mathcal{N}_{G}(1_{G})$ such that $V$ is compact and $VV^{-1}\subseteq U$.
Since $G$ is $\sigma$-compact, there is a sequence
$\{a_{n}\}_{n<\omega}\subseteq G$ with $G=\cup _{n<\omega }\{Va_{n}\}$. Hence,
$H_{0}=\cup_{n<\omega }t^{-1}(Va_{n})$. On the other hand,since $t$ is Bohr continuous
and $Va_{n}$ is a compact subset of $G^{b}$ for all $n<\omega $, it follows
that $t^{-1}(Va_{n})$ is a Bohr-closed subset of $H_{0}^{b}$ and, as a
consequence, a closed subset of $H_0$ for all $n<\omega $. The fact that $H_0$
is of second category in itself implies that there is some $n<\omega $ such
that $t^{-1}(Va_{n})$ has non-empty interior. Then $t^{-1}(VV^{-1})\subseteq
t^{-1}(U)$ contains $1_{H}$ in its interior, what proves the continuity of
$t $. This completes the proof in this case.

Finally let us suppose that $G^b$ is realcompact and $G$ respects compactness.
Since the Bohr topology of any topological group is
always weaker than the original locally compact topology of that group, it follows that $t$
is continuous on the compact subsets of $H_0$ when this group is
equipped with the original locally compact topology inherited from $H$. Thus, if $C$ is any
compact subset of $H_0$, then $t(C)$ is a Bohr compact subset of $G^{b}$. Now,
by hypothesis, the group $G$ respects compactness, therefore,
it holds that $t(C)$ is also compact in the
original locally compact topology of $G$. Moreover, since $H_0$ 
is a topological $k$-space and we have just seen that the map $t$ of
$H_0$ onto $G$ is continuous on the compact subsets of $H_0$,
we deduce that $t$ is continuous
on $H_0$ when with respect to the locally compact topologies of $H$ and $G$.
\end{proof}
Combining the last result and Corollary \ref{co22}, we get the
following consequences for arbitrary $\sigma$-compact $MAPLC$
groups.
\begin{co}
\label{co33}Let $G$ and $H$ be $\sigma$-compact $MAPLC$ groups and
suppose that $T$ is a non-vanishing linear isometry of $AP(G)$
onto $AP(H)$
that respects (finite dimensional)
unitary representations. Then there exists a topological isomorphism $t$ of $H$ onto $G$
and an element
$\gamma \in \widehat{H}$ such that
\begin{equation*}
(Tf)(h)=\gamma (h)~f(t(h))\text{ for all }h\in H_{0}\text{ and for all }f\in
AP(G)\text{.}
\end{equation*}
\end{co}
\begin{co}
\label{co34}If $\tau_{1}$ and $\tau_{2}$ are two $\sigma$-compact
$MAPLC$ topologies on a group $G$ such that
$AP(G,\tau_{1})=AP(G,\tau_{2})$, that is, they have exactly the
same almost periodic functions. Then $\tau_{1}=\tau_{2}$; or,
equivalently, there is at most one $\sigma$-compact maximally
almost periodic locally compact topology compatible with a fixed
Bohr compactification of the group.
\end{co}
\begin{proof}
Clearly, $\tau_{1}$ and $\tau_{2}$ have associated the same Bohr topology, say $\tau_{1}^{b}$ and $\tau_{2}^{b}$. Hence the identity map $t$ of $(G,\tau_{1}^{b})$ onto $(G,\tau_{2}^{b})$ is a topological group isomorphism. Then it is proved as at the end of Theorem \ref{th32} that $t$ is also a topological isomorphism of $(G,\tau_{1})$ onto $(G,\tau_{2})$.
\end{proof}
\begin{co}
\label{co35}If $\tau_{1}$ and $\tau_{2}$ are two maximally almost
periodic locally compact topologies on a group $G$ such that both
respect compactness and $AP(G,\tau_{1})=AP(G,\tau_{2})$. Then
$\tau_{1}=\tau_{2}$.
\end{co}
\begin{proof}
Consider the identity linear operator $T:AP(G,\tau_{1})\longrightarrow AP(G,\tau_{2})$.
The mapping $t:(G,\tau_{2}^{b})\longrightarrow (G,\tau_{1}^{b})$
canonically associated to this isometry (the identity mapping)
is a topological isomorphism with respect to the
Bohr topology of both groups. Now, since we are dealing with locally compact groups that
respect compactness, the arguments in Theorem \ref{th32} apply to obtain that $t$ is a
topological isomorphism with respect to the original topologies. Thus $\tau_{1}=\tau_{2}$.
\end{proof}
\begin{re}
\label{re36}The statement above is not true in general if the requirement
of respecting compactness on the groups is removed.

Indeed, it is known that if $H$ is a compact, connected, simple
Lie group with trivial center, and $h$ is any homomorphism from
$H$ into any compact topological group $K$. Then $h$ is continuous
(cf. \cite{waerden} or \cite[9.16]{co1}.) Thus the group $H$
posses two different group topologies: the discrete topology and
the original one which makes $H$ a compact Lie group. However,
both topologies share exactly the same family of almost periodic
functions: notice that $H$ equipped with the original compact
topology is the Bohr compactification of the group $H$ equipped
with the discrete topology. This sort of examples cannot be
expected for Abelian groups. In fact Varopoulos proved (cf.
\cite{varopou}) that for any Abelian group $G$ there is at most
one locally compact group topology compatible with a fixed Bohr
compactification of the group. Corollary \ref{co34} extends
Varapoulos' result to $\sigma$-compact $MAPLC$ groups.
\end{re}
Furthermore, since $LC$ Abelian groups respect compactness (cf. \cite{glick})
and are always realcompact in their Bohr topology (cf. \cite{cht1}),
we can extend Theorem \ref{th23} and Corollary \ref{co24}
to $LC$ Abelian groups following the pattern of Theorem \ref{th32}.
\begin{theor}
\label{th37}Let $G$ and $H$ be $LCA$ groups and $H$ is connected. Suppose
that $T$ is a non-vanishing linear isometry of $AP(G)$ into $AP(H)$ that preserves
trigonometric polynomials.
Then there is a closed subgroup $H_{0}\subseteq H$, a continuous group homomorphism
$t$ of $H_{0}$ onto $G$,
an element $h_{0}\in H_{0}$, a character $\alpha \in \widehat{H}$ and an
unimodular complex number $a$ such that
\begin{equation*}
(Tf)(h)=a\cdot \alpha (h)~\cdot f(t(h-h_{0}))\text{ for all }h\in H_{0}\text{
and for all }f\in AP(G)\text{.}
\end{equation*}
Moreover, if $T$ is an onto isometry then $H_{0}=H$ and, as a consequence, $G
$ and $H$ are topologically isomorphic.
\end{theor}
\begin{co}
\label{co38}Let $\rho :H\longrightarrow G$ where $G$ and $H$ are $LCA$
groups and $H$ is connected. Then $\rho $ preserves trigonometric
polynomials iff $\rho =t+\theta $ where $t$ is a continuous group homomorphism
and $\theta $ is a constant map.
\end{co}
\begin{re} \label{re39}
In all results in this section we have assumed the hypothesis that
the isometry $T$ is non-vanishing. It is readily
seen that this assumption may not be removed in general; for example, it is
enough to take $H$ as any non compact $LCA$ group and define $G$ as the Bohr
compactification of $H$. The map $Tf=f_{|H}$ for all $f\in AP(G)$ is a
simple multiplicative linear isometry of $AP(G)$ onto $AP(H)$ while the
groups $G$ and $H$ are not homeomorphic. However, this example has no
special interest since one of the groups is never metrizable. Next we give a
less obvious example due to Kunen \cite{kunen1} to show that in general
multiplicative linear onto
isometries between spaces of almost periodic functions are not always
related to continuous 1-1 mappings between the groups involved, even when
they are metrizable.
\end{re}
\begin{example}
\label{ex310}There exists a pair of discrete Abelian groups $G$ and $H$ and a
multiplicative linear isometry $T$ of $AP(G)$ onto $AP(H)$ that may not be
represented by means of a continuous 1-to-1 mapping of $H$ into $G$.
\end{example}
\begin{proof}
For every prime $p$, let $\mathbb{V}_{p}$ be the vector space over
$\mathbb{Z}_{p}$ of dimension $\aleph _{0}$ where we just consider
$\mathbb{V}_{p}$ as an (additive) Abelian group, and ignore the
vector space structure. It is shown in \cite[(4.2)]{kunen1} that
for distinct primes, $p$ and $q$, the groups $\mathbb{V}_{p}$ and
$\mathbb{V}_{q}$ equipped with their respective Bohr topologies
are not homeomorphic; in fact, there is no 1-1 Bohr-continuous
mapping from $\mathbb{V}_{p}$ into $\mathbb{V}_{q}$. On the other
hand, their respective Bohr compactifications, $b\mathbb{V}_{p}$
and $b\mathbb{V}_{q}$, are homeomorphic (cf. \cite[(9.15)]{hr1}).
Let $t$ be a homeomorphism of $b\mathbb{V}_{q}$ onto
$b\mathbb{V}_{p}$. If we take $G=\mathbb{V}_{p}$,
$H=\mathbb{V}_{q}$ and define $T:AP(G)\longrightarrow AP(H)$ by
$T(f)=f^{b}\circ t$ where $f^{b}$ means the continuous extension
to $bG$ of the mapping $f$ then it is easily verified that $T$ may
not be represented by a 1-1 continuous mapping of $H$ into $G$.
\end{proof}
\begin{re}
We do not know whether there is a variant of our results for
spaces of weakly almost periodic functions. In connection with
this question, observe that each isometry $T$, defined between
spaces of continuous functions, that preserves finite dimensional
unitary representations  sends automatically almost periodic
functions into almost periodic functions. Thus, nothing new may be
obtained along this way unless one replaces the preservation of
finite dimensional representations by a weaker condition.
\end{re}

We wish to thank the referee for his/her constructive report that
has helped us to improve parts of this paper.

\end{document}